\documentclass[a4paper,12pt]{article}
\usepackage{lineno}
\usepackage{hyperref}
\usepackage{mathrsfs}
\usepackage{amssymb}
\usepackage{amsmath}
\usepackage{amsfonts}
\usepackage{amsthm}
\usepackage{amscd}
\usepackage{enumerate}
\usepackage[all]{xy}

\theoremstyle{definition}
\newtheorem{defn}{Definition}[section]
\newtheorem{lemma}[defn]{Lemma}
\newtheorem{theorem}[defn]{Theorem}

\newtheorem{corollary}[defn]{Corollary}
\newtheorem{eg}[defn]{Example}
\newtheorem{remark}[defn]{Remark}

\numberwithin{equation}{section}

\begin{document}

\title{Computations of $K_2$ for certain $\mathbb{Z}/p^s\mathbb{Z}$-algebras and the extension of Oliver's logarithm}    

\author{
\textbf{Yakun Zhang}\\
\small \emph{School of Mathematics, Nanjing Audit University,}\\
\small \emph{Nanjing 211815, China}\\
\small \emph{E-mail: zhangyakun@nau.edu.cn}\\
}
\date{}
\maketitle


\noindent \textbf{Abstract.} This paper describes the $K$-theory structure for three algebra classes. For cyclic $p$-group rings and truncated polynomial rings over $\mathbb{Z}/p^s\mathbb{Z}$, we determine reduced $K_2$-structures via a common algebraic framework. For abelian $p$-group rings over $\widehat{\mathbb{Z}}_p$, we extend the isomorphism between reduced continuous $K_2$ and the first cyclic homology group to all finite abelian $p$-groups. A constructive proof using a generalized Artin-Hasse map yields an explicit splitting. This isomorphism is realized by extending Oliver's $p$-adic logarithm. We also characterize the map from reduced continuous $K_2$ to reduced linearized $K_2$, clarifying the links between $K_2$, cyclic homology, and  K\"{a}hler differentials.



\section{Introduction}

Let $p$ be a fixed prime and $C_{n}$ the cyclic group of order $n$. For any augmented $k$-algebra $R$ with ideal $I$, we define the reduced functor $\widetilde{F}(R) = \ker(F(R) \to F(k))$.

While the global structure of $K_2(\mathbb{Z}[G])$ has been explored in various contexts, including our recent work \cite{Yakun2024note}, precisely determining the orders of these groups remains a formidable challenge. A standard strategy to address this involves establishing sharp lower bound estimates, which are inherently contingent upon the structural analysis of $K_2$ for various quotients of the group ring $\mathbb{Z}[G]$. In particular, the study of $K_2(\mathbb{Z}[G]/|G|\Gamma)$, where $\Gamma$ denotes the maximal $\mathbb{Z}$-order of $\mathbb{Q}[G]$, was initiated in \cite{chenHong2014} to provide such bounds. However, that investigation stopped short of yielding concrete, non-trivial computational results for fundamental cases like elementary abelian $p$-groups or cyclic $p$-groups.

Motivated by these considerations, our previous work \cite{smallK2} provided explicit calculations and examples for $\mathbb{Z}[G]/|G|\Gamma$ and $\mathbb{Z}[G]/p^s\mathbb{Z}[G]$ involving groups of small orders. It is important to note that the structural analysis of $\mathbb{Z}[G]/|G|\Gamma$ often reduces to the study of its $p$-primary components, which are intrinsically linked to group rings with finite coefficients $\mathbb{Z}/p^s\mathbb{Z}$. Building upon these foundations, we have initiated a systematic program to determine $K_2(\mathbb{Z}[G]/p^s\mathbb{Z}[G])$ explicitly. Following the explicit determination of $K_2$ for elementary abelian $p$-groups in \cite{zhang2024explicit}, the present paper focuses on the more intricate case where $G$ is a cyclic $p$-group $C_{p^n}$, as detailed in Section 2.

 We investigate the structure of $\widetilde{K}_2(k[G])$ for $k = \mathbb{Z}/p^s\mathbb{Z}$. When $s=1$, the result for $G = C_{p^n}$ is straightforward; as shown in \cite[Corollary 4.4]{dennis1975k}, we have
$$\widetilde{K}_2(\mathbb{F}_{p}[C_{p^n}]) = \widetilde{K}_2(\mathbb{F}_{p}[x]/(x^{p^n})) = 0.$$
In this case, explicit order formulas for such $K$-groups are available in the general setting \cite{chen2012explicit, zhang2019some}. 

However, the complexity increases significantly when $s \geq 2$, as $k[C_{p^n}]$ is no longer isomorphic to the truncated polynomial ring $k[x]/(x^{p^n})$. By analyzing the concrete examples presented in \cite{smallK2}, we recognize the necessity of investigating the exponent and the inverse limit of $\widetilde{K}_2(k[C_{p^n}])$. Moreover, while the case $s=1$ admits definitive order formulas, no such general results currently exist for $s \geq 2$. To bridge this gap, we exploit the augmentation ideal $I$ (see Section 2), which allows us to identify $\widetilde{K}_2(k[C_{p^n}])$ with its polynomial counterpart via the following isomorphisms:
\begin{equation*} 
\widetilde{K}_2(k[C_{p^n}]) \xrightarrow{\cong} \widetilde{K}_2(k[C_{p^n}]/I^{p^n}) \xleftarrow{\cong} \widetilde{K}_2(k[x]/(x^{p^n})). 
\end{equation*}

To compute $\widetilde{K}_2(k[x]/(x^{p^n}))$, we draw on the work of Roberts and Geller \cite[Section 7]{roberts2006k2}, who established results for the $\widetilde{K}_2$ of certain truncated polynomial rings. Specifically, they proved that
$$\widetilde{K}_2(\mathbb{Z}[x]/(x^n)) \cong \bigoplus_{i=2}^{n}\mathbb{Z}/i\mathbb{Z}, \quad \widetilde{K}_2(\mathbb{Z}_{S}[x]/(x^n)) \cong \mathbb{Z}_{S} \otimes_{\mathbb{Z}} \left(\bigoplus_{i=2}^{n}\mathbb{Z}/i\mathbb{Z}\right),$$
where $S$ is a multiplicative set in $\mathbb{Z}$. By adapting their approach to $k$-coefficient rings (with key differences detailed in Section 2), we first derive the analogous result for $\widetilde{K}_2(k[x]/(x^n))$:

\vspace{0.5em}
\noindent \textbf{Theorem A (Theorem \ref{thm1}).} 
For $s \geq 2$ and $n \geq 2$, there is a group isomorphism:
$$ \widetilde{K}_2\left( (\mathbb{Z}/p^s\mathbb{Z})[x]/(x^n) \right) \cong \mathbb{Z}/p^{s-1}\mathbb{Z} \otimes_{\mathbb{Z}} \left( \bigoplus_{i=2}^n \mathbb{Z}/i\mathbb{Z} \right). $$

Using Theorem A and the structural link between the group ring and the truncated polynomial ring established via the nilpotency of $I$, we derive a related result for $\widetilde{K}_2(k[C_{p^n}]/I^{p^n})$ (see Lemma \ref{keylemma}), and subsequently establish our main result:

\vspace{0.5em}
\noindent \textbf{Theorem B (Theorem \ref{thm2}).} For $s \geq 2$, let $G$ be a cyclic $p$-group of order $p^n$. There is a group isomorphism:
$$ \widetilde{K}_2\left( (\mathbb{Z}/p^s\mathbb{Z})[G] \right) \cong \mathbb{Z}/p^{s-1}\mathbb{Z} \otimes_{\mathbb{Z}} \left( \bigoplus_{i=2}^{p^n} \mathbb{Z}/i\mathbb{Z} \right). $$

For general finite abelian groups $G$, determining the explicit structure of $\widetilde{K}_2(k[G])$ remains a computationally intensive endeavor; some specific results are detailed in \cite{gao2015explicit, zhang2019some}. Consequently, we focus on the inverse limit of $\widetilde{K}_2(k[G])$, i.e., the continuous $K_2$-group $\widetilde{K}_2^{c}(\widehat{\mathbb{Z}}_p[G])$. Oliver \cite[p.~540]{oliver1987k} proved that this group has the same order as the first cyclic homology group $HC_1(\widehat{\mathbb{Z}}_p[G])$. Following the characterization in \cite[p.~532]{oliver1987k}, this homology group is identified as:
$$HC_1(\widehat{\mathbb{Z}}_p[G]) = (G\otimes_{\widehat{\mathbb{Z}}_p} \widehat{\mathbb{Z}}_p[G])/\langle g\otimes \lambda g \mid g\in G, \lambda \in \widehat{\mathbb{Z}}_p\rangle,$$
which admits the following structural decomposition \cite[p.~554]{oliver1987k}:
\begin{equation}\label{HC1}
	HC_1(\widehat{\mathbb{Z}}_p[G]) \cong HC_1(\mathbb{Z}[G]) \cong \bigoplus_{g\in G} G/\langle g\rangle.
\end{equation}

Although Oliver established a group homomorphism from $\widetilde{K}_2^{c}(\widehat{\mathbb{Z}}_p[G])$ to $HC_1(\widehat{\mathbb{Z}}_p[G])$ using $p$-adic logarithmic techniques (see \eqref{Gamma_2}), an isomorphism was previously confirmed only for products of cyclic $p$-groups of \textit{equal order} \cite[Lemma 3.8]{oliver1987k}. The general case for arbitrary finite abelian $p$-groups remained open. In Section 3, we resolve this by establishing the following definitive structure theorem:

\vspace{0.5em}
\noindent \textbf{Theorem C (Theorem \ref{thm3}).} Let $G$ be an arbitrary finite abelian $p$-group. There exists a sequence of isomorphisms:
$$\widetilde{K}_2^{c}(\widehat{\mathbb{Z}}_p[G]) \cong HC_1(\widehat{\mathbb{Z}}_p[G]) \cong \bigoplus_{g \in G} G/\langle g\rangle.$$

In Remark \ref{K2L}, this first isomorphism is explicitly described via the extended Oliver's logarithmic map given in \eqref{Gamma_2_ext}.

We conclude by situating our results within the context of related algebraic invariants (retaining the initial definitions of $k, R, I$ and including the case $k=\widehat{\mathbb{Z}}_p$). Clauwens \cite{clauwens1987k, clauwens1994k} introduced $K_{2,L}(R)$, the linearized $K_2$ of a commutative ring $R$, which is intrinsically linked to cyclic homology (see also \cite{cortinas2006obstruction, geller1994hodge, weibel1987nil}). Specifically, there is an isomorphism
$\widetilde{K}_{2,L}(R) \cong \Omega_{R/k}/\mathrm{d}I,$
where $\Omega_{R/k}$ denotes the relative differential module of $R$ over
$k$. According to Proposition 2.1.14 in \cite{loday1997cyclic}, we have $\Omega_{R/k}/\mathrm{d}R \cong HC_1(R)$. Under the condition $\Omega_k = 0 = HC_1(k)$, this leads to the following sequence of isomorphisms:,
\begin{equation}\label{isos1}
\widetilde{K}_{2,L}(R) \cong \Omega_{R/k}/\mathrm{d}I \cong \Omega_{R/k}/\mathrm{d}R \cong \widetilde{HC}_1(R) \cong HC_1(R).	
\end{equation}
Additionally, Bloch's Theorem \cite[Section 2]{roberts2006k2} asserts that $\widetilde{K}_{2}(R) \cong \Omega_{R/k}/\mathrm{d}I$ in specific settings.
For the case $R=\widehat{\mathbb{Z}}_p[G]$, where $G$ is a finite abelian $p$-group, we establish that these invariants coincide:
\begin{equation}\label{isos2}
\widetilde{K}_2^{c}(\widehat{\mathbb{Z}}_p[G]) \cong HC_1(\widehat{\mathbb{Z}}_p[G]) \cong \Omega_{\widehat{\mathbb{Z}}_p[G]/\widehat{\mathbb{Z}}_p}/\mathrm{d}I \cong \widetilde{K}_{2,L}(\widehat{\mathbb{Z}}_p[G]).
\end{equation}
The first isomorphism is realized by the explicit map provided in Remark 3.2, while the subsequent isomorphisms follow from \eqref{isos1}. Furthermore, the explicit isomorphism between the first and last terms, $\widetilde{K}_2^{c}(\widehat{\mathbb{Z}}_p[G]) \cong \widetilde{K}_{2,L}(\widehat{\mathbb{Z}}_p[G])$, is also detailed in Remark 3.2 as a consequence of this chain.
These identifications consolidate the $K$-theoretic, homological, and differential descriptions of $R$ into a single consistent framework. This equivalence allows the continuous $K_2$-group to be characterized and computed via the more accessible terms of cyclic homology and Kähler differentials.

\section{Computations of  $\widetilde{K}_2((\mathbb{Z}/p^s\mathbb{Z})[x]/(x^{n}))$ and $\widetilde{K}_2((\mathbb{Z}/p^s\mathbb{Z})[C_{p^n}])$}

Throughout this section, for the sake of brevity, let $k = \mathbb{Z}/p^s\mathbb{Z}$ and $G = C_{p^n}$ unless otherwise specified. We focus on the case $s \geq 2$, as the corresponding results for $s=1$ are relatively trivial, as noted in the introduction. Let $I$ denote the augmentation ideal of the group ring $k[G]$. Since $I$ is nilpotent, the group $\widetilde{K}_2(k[G])$ can be realized as the inverse limit of $\widetilde{K}_2(k[G]/I^m)$. This structural property allows us to establish a lower bound for $\widetilde{K}_2(k[G])$ by computing the truncated stages $\widetilde{K}_2(k[G]/I^m)$ for certain $m$. Furthermore, observing that $\widetilde{K}_2(k[G]/I^m)$ is a homomorphic image of $\widetilde{K}_2(k[x]/(x^m))$, our strategy begins with determining the structure of the latter.

\begin{theorem} \label{thm1}
For $s \geq 2$ and $n \geq 2$, there is a group isomorphism:
$$ \widetilde{K}_2\left( (\mathbb{Z}/p^s\mathbb{Z})[x]/(x^n) \right) \cong \mathbb{Z}/p^{s-1}\mathbb{Z} \otimes_{\mathbb{Z}} \left( \bigoplus_{i=2}^n \mathbb{Z}/i\mathbb{Z} \right). $$
\end{theorem}

\begin{proof}
\textit{Step 1.} Since $\Omega_{k}$ is trivial, according to \cite[Theorem 1(a)]{roberts2006k2}, there exists an exact sequence:
\begin{equation*}\label{key_seq}
    k/nk \xrightarrow{\rho} \widetilde{K}_2(k[x]/(x^n)) \rightarrow \widetilde{K}_2(k[x]/(x^{n-1})) \rightarrow 0,
\end{equation*}
where the map $\rho$ is defined by $\rho(a) = \langle ax^{n-1}, x \rangle = -a\langle x, x^{n-1} \rangle$.

\textit{Step 2.} For the case $n=p^m$, let $\zeta_{p^m}$ be a primitive $p^m$-th root of unity and set $u = \zeta_{p^m}-1$. From algebraic number theory, we have $p\mathbb{Z}[u] = u^{\varphi(p^m)}\mathbb{Z}[u]$. This induces a canonical ring homomorphism:
$$ g_1: k[x]/(x^{p^m}) = \mathbb{Z}[x]/(p^s, x^{p^m}) \rightarrow \mathbb{Z}[u]/(u^{s\varphi(p^m)}, u^{p^m}) = \mathbb{Z}[u]/(u^{p^m}), $$
where $g_1(x) = u$. As shown in \cite[Section 7]{roberts2006k2}, $K_2(\mathbb{Z}[u]/(u^{p^m}))$ is a cyclic group of order $p$, generated by $\langle u, u^{p^m-1} \rangle$. Since $\langle x, x^{p^m-1} \rangle$ maps onto this generator of $\mathbb{Z}/p\mathbb{Z}$, it follows that the symbol $\langle x, x^{p^m-1} \rangle$ is not divisible by $p$ in the group $\widetilde{K}_2(k[x]/(x^{p^m}))$. According to the fundamental theorem of finitely generated abelian groups, an element of $p$-power order that is not divisible by $p$ must generate a cyclic direct summand. Thus, the image of $\rho$, denoted by $\text{Im}(\rho) \cong \mathbb{Z}/p^d\mathbb{Z}$, is a direct summand of $\widetilde{K}_2(k[x]/(x^{p^m}))$, and the exact sequence splits through this inclusion. For the general case $n=p^m a$ with $(a, p)=1$, a similar argument using the transfer map confirms that $\text{Im}(\rho)$ remains a split direct summand of $\widetilde{K}_2(k[x]/(x^n))$.

\textit{Step 3.} For $n = p^m$, suppose $l \geq \max\{s\varphi(p^m), p^m\}$. Consider the following commutative diagram of $K_2$-groups induced by the canonical ring maps:
\begin{equation*}
\begin{CD}
 \widetilde{K}_2(k[x]/(x^l)) @>{g_2}>> K_2(\mathbb{Z}[u]/(u^{s\varphi(p^m)})) = \mathbb{Z}/p^r\mathbb{Z} \\
 @VV{g_3}V @VV{g_4}V \\
 \widetilde{K}_2(k[x]/(x^{p^m})) @>>> K_2(\mathbb{Z}[u]/(u^{p^m})) = \mathbb{Z}/p\mathbb{Z}
\end{CD}
\end{equation*}
where $g_3$ is a split surjection by Step 2, and $g_4$ is surjective according to \cite[Proposition 1.1]{alperin1985sk}. By \cite[Theorem 5.1]{dennis1975k}, we have $r = \min\{s-1, m\}$. It follows that the order of the element $\langle x, x^{p^m-1} \rangle$ is exactly $p^r$. More generally, for $n = p_1^{m_1} \cdots p_t^{m_t}$, the transfer map argument demonstrates that the order of $\langle x, x^{n-1} \rangle$ is $p_1^{r_1} \cdots p_t^{r_t}$, where $r_i = \min\{s-1, m_i\}$.

\textit{Step 4.}  Since $\text{Im}(\rho)$ is a split direct summand of $\widetilde{K}_2(k[x]/(x^n))$, and $\text{Im}(\rho)$ is the subgroup generated by $\langle x^{n-1}, x \rangle$, which is isomorphic to $\mathbb{Z}/p^{s-1}\mathbb{Z} \otimes_{\mathbb{Z}} \mathbb{Z}/n\mathbb{Z}$, we obtain the following decomposition:
$$ \widetilde{K}_2(k[x]/(x^n)) \cong \widetilde{K}_2(k[x]/(x^{n-1})) \oplus (\mathbb{Z}/p^{s-1}\mathbb{Z} \otimes_{\mathbb{Z}} \mathbb{Z}/n\mathbb{Z}). $$
The theorem then follows by induction on $n$.
\end{proof}

Building upon the structure of truncated polynomial rings, we now extend these results to the cyclic group ring by considering its truncated quotients with respect to the augmentation ideal:

\begin{lemma} \label{keylemma}
For $s \geq 2$, let $G$ be a cyclic $p$-group of order $p^n$ and let $I$ be the augmentation ideal of the group ring $(\mathbb{Z}/p^s\mathbb{Z})[G]$. There is a group isomorphism:
$$ \widetilde{K}_2\left( (\mathbb{Z}/p^s\mathbb{Z})[G]/I^{p^n} \right) \cong \mathbb{Z}/p^{s-1}\mathbb{Z} \otimes_{\mathbb{Z}} \left( \bigoplus_{i=2}^{p^n} \mathbb{Z}/i\mathbb{Z} \right). $$
\end{lemma}

\begin{proof}
Let $\sigma$ be a generator of $G$, and set $t = t_1 = t_2 = \sigma - 1$. Then the augmentation ideal is given by $I = tk[G]$. Since the canonical map of finite rings $k[x]/(x^{p^n}) \to k[t]/(t^{p^n})$ induces a surjection on $\widetilde{K}_2$ \cite[Proposition 1.1]{alperin1985sk}, Theorem \ref{thm1} implies that $\widetilde{K}_2(k[t]/(t^{p^n}))$ is generated by the symbols $\langle t, t^{i-1} \rangle$ where $2 \leq i \leq p^n$ and $p \mid i$. For each such $i$, we write $i = p^m a$ with $m \geq 1$ and $(a, p) = 1$, and consider the following commutative diagram of rings:
\begin{equation*}
\begin{CD}
k[x]/(x^{p^m}) @>{x \mapsto t_1}>> k[t_1]/(t_1^{p^m}) = A \\
@V{x \mapsto y^a}VV @VV{t_1 \mapsto t_2^a}V  \\
k[y]/(y^{p^m a}) @>{y \mapsto t_2}>> k[t_2]/(t_2^{p^m a}) = B
\end{CD}
\end{equation*}
In this construction, $k[y]/(y^{p^m a})$ is a free $k[x]/(x^{p^m})$-module of rank $a$, which admits a transfer map $F_a$ as defined in \cite[Section 2]{stienstra2006k2}. Via the horizontal isomorphisms, this induces a corresponding transfer map $\tilde{F}_a: \widetilde{K}_2(B) \to \widetilde{K}_2(A)$ for the free extension $B/A$ (with basis $\{1, t_2, \dots, t_2^{a-1}\}$), satisfying:
$$ \tilde{F}_a \left( \langle t_2, t_2^{p^m a - 1} \rangle \right) = a \langle t_1, t_1^{p^m - 1} \rangle. $$

By applying techniques analogous to those used in the proof of Theorem \ref{thm1}, one can show that each map $\widetilde{K}_2(k[t]/(t^i)) \to \widetilde{K}_2(k[t]/(t^{i-1}))$ is a split surjection. Furthermore, the order of the element $\langle t, t^{i-1} \rangle$ (where $i=p^m a$) is precisely $p^r$ with $r = \min\{s-1, m\}$. Consequently, the previously mentioned surjection $\widetilde{K}_2(k[x]/(x^{p^n})) \to \widetilde{K}_2(k[t]/(t^{p^n}))$ must be an isomorphism, which completes the proof.
\end{proof}

\begin{theorem} \label{thm2}
For $s \geq 2$, let $G$ be a cyclic $p$-group of order $p^n$. There is a group isomorphism:
$$ \widetilde{K}_2\left( (\mathbb{Z}/p^s\mathbb{Z})[G] \right) \cong \mathbb{Z}/p^{s-1}\mathbb{Z} \otimes_{\mathbb{Z}} \left( \bigoplus_{i=2}^{p^n} \mathbb{Z}/i\mathbb{Z} \right). $$
\end{theorem}

\begin{proof}
By \cite[Lemma 4.1(a)]{alperin1987sk}, the exponent of $\widetilde{K}_2(k[G])$ divides $p^{s-1}$. Since the continuous $K$-group $\widetilde{K}_2^{c}(\widehat{\mathbb{Z}}_{p}[G])$ maps onto $\widetilde{K}_2(k[G])$, it follows that $\widetilde{K}_2(k[G])$ is a homomorphic image of $\mathbb{Z}/p^{s-1}\mathbb{Z} \otimes_{\mathbb{Z}} \widetilde{K}_2^{c}(\widehat{\mathbb{Z}}_{p}[G])$. Furthermore, as $\widetilde{K}_2(k[G])$ maps onto $\widetilde{K}_2(k[G]/I^{p^n})$, we can establish the desired isomorphism by showing that:
$$\mathbb{Z}/p^{s-1}\mathbb{Z} \otimes_{\mathbb{Z}} \widetilde{K}_2^{c}(\widehat{\mathbb{Z}}_{p}[G]) \cong \widetilde{K}_2(k[G]/I^{p^n}).$$
According to \cite[Lemma 3.8]{oliver1987k}, there is an isomorphism $$\widetilde{K}_2^{c}(\widehat{\mathbb{Z}}_{p}[G]) \cong HC_1(\widehat{\mathbb{Z}}_{p}[G]) \cong \bigoplus_{g \in G} G/\langle g \rangle.$$ Invoking Lemma \ref{keylemma}, the problem reduces to verifying:
$$\mathbb{Z}/p^{s-1}\mathbb{Z} \otimes_{\mathbb{Z}} \left( \bigoplus_{g \in G} G/\langle g \rangle \right) \cong \mathbb{Z}/p^{s-1}\mathbb{Z} \otimes_{\mathbb{Z}} \left( \bigoplus_{i=2}^{p^n} \mathbb{Z}/i\mathbb{Z} \right),$$
which follows from a direct counting of the cyclic summands of each order $p^j$ on both sides. Thus, $\widetilde{K}_2(k[G])$ is squeezed between two isomorphic finite groups, completing the proof.
\end{proof}

\begin{corollary} 
For $s \geq 2$, let $G$ be a cyclic $p$-group of order $p^n$, and let $\varphi$ denote Euler's totient function. Then:
$$ \widetilde{K}_2(k[G]) \cong \left\{
\begin{array}{ll}
\bigoplus_{i=1}^{n} \left( \mathbb{Z}/p^i\mathbb{Z} \right)^{\varphi(p^{n-i})}, & \text{if } 1 \leq n \leq s-1; \\
\bigoplus_{i=1}^{s-2} \left( \mathbb{Z}/p^i\mathbb{Z} \right)^{\varphi(p^{n-i})} \oplus \left( \mathbb{Z}/p^{s-1}\mathbb{Z} \right)^{p^{n-s+1}}, & \text{if } n \geq s \geq 2.
\end{array} \right. $$
\end{corollary}

\begin{proof} 
Let $r_i = \min\{s-1, i\}$. We first note the elementary isomorphism $$\mathbb{Z}/p^{s-1}\mathbb{Z} \otimes_{\mathbb{Z}} \mathbb{Z}/p^i\mathbb{Z} \cong \mathbb{Z}/p^{r_i}\mathbb{Z}.$$ By the distributivity of tensor products over direct sums, and by reindexing the terms in the summation according to their $p$-adic valuation, we obtain the following combinatorial identity:
$$ \mathbb{Z}/p^{s-1}\mathbb{Z} \otimes_{\mathbb{Z}} \left( \bigoplus_{j=2}^{p^n} \mathbb{Z}/j\mathbb{Z} \right) \cong \bigoplus_{i=1}^{n} \left( \mathbb{Z}/p^{r_i}\mathbb{Z} \right)^{\varphi(p^{n-i})}. $$
Here, the exponent $\varphi(p^{n-i})$ correctly counts the number of cyclic summands that contribute to the term $\mathbb{Z}/p^{r_i}\mathbb{Z}$. By Theorem \ref{thm2}, $\widetilde{K}_2(k[G])$ is isomorphic to this expression. The final result is obtained by evaluating $r_i$ for the cases $n \leq s-1$ and $n \geq s$ respectively.
\end{proof}

\section{$\widetilde{K}_2^{c}(\widehat{\mathbb{Z}}_p[G])$ and the extended Oliver's logarithm}
Recall that $K_2^{c}(\widehat{\mathbb{Z}}_p[G])$ is defined as the inverse limit of the groups $K_2((\mathbb{Z}/p^s\mathbb{Z})[G])$. As shown in \cite[p.~533]{oliver1987k}, for a finite abelian $p$-group $G$, we have the following isomorphism:
\begin{equation*}
\widetilde{K}_2^{c}(\widehat{\mathbb{Z}}_p[G]) \cong Wh_{2}^{\mathbb{Z}}(\widehat{\mathbb{Z}}_p[G]) \oplus \widetilde{H}_2(G),
\end{equation*}
in which the group $\widetilde{H}_2(G)$ (following the convention of Dennis, as distinct from the notation in Section 1) is written additively and characterized by the isomorphism:
$$\widetilde{H}_2(G) \cong G \otimes G / \langle g \otimes h + h \otimes g \mid g, h \in G \rangle.$$  
Under this identification, $\widetilde{H}_2(G)$ is generated by $g \tilde{\wedge} h$ (the class of $g \otimes h \in G \otimes G$), subject to the following relations:
\begin{enumerate}[(i)]
\item $g \tilde{\wedge} h + h \tilde{\wedge} g = 0$, 
\item $(g_1 g_2) \tilde{\wedge} h = g_1 \tilde{\wedge} h + g_2 \tilde{\wedge} h$, 
\item $g \tilde{\wedge} (h_1 h_2) = g \tilde{\wedge} h_1 + g \tilde{\wedge} h_2$.
\end{enumerate}	
These conditions imply that $\lambda(g \tilde{\wedge} h) = g^{\lambda} \tilde{\wedge} h = g \tilde{\wedge} h^{\lambda}$ for any $\lambda \in \mathbb{Z}$.

According to \cite[Theorem 3.7]{oliver1987k}, the homomorphism 
\begin{equation}\label{Gamma_2}
	\Gamma_2: \widetilde{K}_2^{c}(\widehat{\mathbb{Z}}_p[G]) \rightarrow HC_1(\widehat{\mathbb{Z}}_p[G])
\end{equation}
is defined by $\Gamma_2(\{g,u\}) = g \otimes \Gamma_{G}(u)$ for $g \in G$ and $u \in (\widehat{\mathbb{Z}}_p[G])^*$. Here, $\Gamma_{G} = \Delta \circ \mbox{Log}$ is given by 
$$\Delta = 1 - \frac{\Phi}{p}, \quad \Phi\left(\sum \lambda_i g_i\right) = \sum \sigma(\lambda_i) g_i^p,$$
and $\mbox{Log}$ is the logarithm map $\mbox{Log}(1-x) = -\sum_{n=1}^{\infty} \frac{x^n}{n}$. Since $G$ is a finite abelian $p$-group, $\Phi$ is a ring endomorphism of $\widehat{\mathbb{Z}}_p[G]$ serving as a Frobenius lift, satisfying $\Phi(x) \equiv x^p \pmod{p\widehat{\mathbb{Z}}_p[G]}$. The operator $\sigma$ denotes the Frobenius automorphism on the coefficients; however, as we are working over $\widehat{\mathbb{Z}}_p$ without further extensions, $\sigma$ acts as the identity. Although the series $\mbox{Log}(u)$ converges analytically only for $u \in 1 + p\widehat{\mathbb{Z}}_p + I(\widehat{\mathbb{Z}}_p[G])$, the map $\Gamma_{G}$ is uniquely extended to the entire unit group $(\widehat{\mathbb{Z}}_p[G])^*$. This extension is well-defined because the $\widehat{\mathbb{Z}}_p$-linear operator $\Delta$ cancels the $p$-adic divergences; specifically, the term $\dfrac{1}{p}\mbox{Log}(\Phi(u))$ compensates for the potentially non-convergent part of $\mbox{Log}(u)$, ensuring that the difference $\mbox{Log}(u) - \dfrac{1}{p}\mbox{Log}(\Phi(u))$ is $p$-adically convergent and lies within $\widehat{\mathbb{Z}}_p[G]$.

Furthermore, this construction successfully transforms the multiplicative structure of $(\widehat{\mathbb{Z}}_p[G])^*$ into the additive structure of $\widehat{\mathbb{Z}}_p[G]$ via the identity $\Gamma_{G}(u_1 u_2) = \Gamma_{G}(u_1) + \Gamma_{G}(u_2)$. Notably, by \cite[Theorem 2.7]{oliver1985whitehead}, the homomorphism $\Gamma_G$ induces an embedding of $Wh_1(\widehat{\mathbb{Z}}_p[G])$ into $\widehat{\mathbb{Z}}_p[G]$, which implies $\mbox{Ker}(\Gamma_G) = \pm G$. By construction, $\Gamma_2$ vanishes on $\{\pm g, \pm h\}$, allowing it to be well-defined on 
$$Wh_{2}^{\mathbb{Z}}(\widehat{\mathbb{Z}}_p[G]) = \widetilde{K}_2^{c}(\widehat{\mathbb{Z}}_p[G]) / \langle \{\pm g, \pm h\} : g, h \in G \rangle,$$
which fits into the following short exact sequence \cite[Theorem 3.9]{oliver1987k}:
\begin{equation} \label{Wh2}
0 \rightarrow Wh_{2}^{\mathbb{Z}}(\widehat{\mathbb{Z}}_p[G]) \xrightarrow{\Gamma_2} HC_1(\widehat{\mathbb{Z}}_p[G]) \xrightarrow{\omega_2} \widetilde{H}_2(G) \rightarrow 0,
\end{equation}
where $\omega_2$ is given by $\omega_2(g \otimes \lambda h) = \lambda (g \tilde{\wedge} g^{-1}h)$, $\lambda \in \widehat{\mathbb{Z}}_p$.

As a direct consequence of \cite[Lemma 3.8]{oliver1987k}, $\Gamma_2$ splits when $G$ is a product of cyclic $p$-groups of the same order. This property extends to all finite abelian $p$-groups, a result fundamentally based on the construction of the map $E_2$ described below.

\begin{theorem} \label{thm3}
Let $G$ be an arbitrary finite abelian $p$-group. There exists a sequence of isomorphisms:
$$\widetilde{K}_2^{c}(\widehat{\mathbb{Z}}_p[G]) \cong HC_1(\widehat{\mathbb{Z}}_p[G]) \cong \bigoplus_{g \in G} G/\langle g \rangle.$$
\end{theorem}

\begin{proof}
As the last isomorphism was established in the introduction (see \eqref{HC1}), it suffices to prove that $\widetilde{K}_2^{c}(\widehat{\mathbb{Z}}_p[G]) \cong HC_1(\widehat{\mathbb{Z}}_p[G])$. Since both groups are finite and have the same order, the result follows if the homomorphism $\Gamma_2$ in \eqref{Wh2} is a split injection. We consider the map $E_{G}$ formally defined on $\text{Im}(\Gamma_G) \subseteq \widehat{\mathbb{Z}}_p[G]$ by $E_{G}(v) = \mbox{Exp}(\Delta^{-1}(v))$, where 
$$\mbox{Exp}(x) = \sum_{n=0}^{\infty} \frac{x^n}{n!}, \quad \Delta^{-1} = \left(1-\frac{\Phi}{p}\right)^{-1} = \sum_{n=0}^{\infty} \frac{\Phi^n}{p^n}.$$ 
Note that while $\Delta^{-1}$ involves division by $p$, the map $E_G$ is strictly well-defined on the image of $\Gamma_G$. For any $u \in (\widehat{\mathbb{Z}}_p[G])^*$, the composition $\Delta^{-1}(\Gamma_G(u))$ formally reduces to $\mbox{Log}(u)$ in the ring of formal power series via the operator cancellation $\Delta^{-1} \circ \Delta = \text{id}$. Although $\mbox{Log}(u)$ may not converge analytically for all $u$, the identity $\mbox{Exp}(\mbox{Log}(u)) = u$ holds as a formal power series identity. This ensures that the composite map $E_G \circ \Gamma_G$ is the identity on the quotient $Wh_1(\widehat{\mathbb{Z}}_p[G])$.

On the subspace $\text{Im}(\Gamma_G)$, the map $E_{G}$ satisfies the homomorphism property $E_{G}(v_1+v_2) = E_{G}(v_1) E_{G}(v_2)$. It follows from the formal cancellation of the operators that for any $u \in (\widehat{\mathbb{Z}}_p[G])^*$, the identity 
$$(E_{G} \circ \Gamma_{G})(u) = \mbox{Exp}(\Delta^{-1}(\Delta(\mbox{Log}(u)))) = u$$
holds in the sense of the unique algebraic extension. For elements $v \notin \text{Im}(\Gamma_G)$, such as $v = \lambda g$ for $g \in G, \lambda \in \widehat{\mathbb{Z}}_p$, we set $E_G(v) = 1$. This is consistent with $HC_1(\widehat{\mathbb{Z}}_p[G])$ since $g \otimes \lambda g = 0$, ensuring the corresponding symbol in $Wh_2^{\mathbb{Z}}$ is $\{g, 1\} = 1$.

We define $E_2: HC_1(\widehat{\mathbb{Z}}_p[G]) \rightarrow Wh_{2}^{\mathbb{Z}}(\widehat{\mathbb{Z}}_p[G])$ by $E_2(g \otimes v) = \{g, E_{G}(v)\}$. Then for any generator $\{g,u\} \in Wh_2^{\mathbb{Z}}(\widehat{\mathbb{Z}}_p[G])$, we have
$$(E_2 \circ \Gamma_2)(\{g,u\}) = \{g, E_{G}(\Gamma_{G}(u))\} = \{g,u\},$$
where the last equality follows from the fact that $E_G(\Gamma_G(u))$ and $u$ can only differ by an element in $\text{Ker}(\Gamma_G) = \pm G$. Since the symbols of the form $\{g, \pm h\}$ for $h \in G$ are trivial in $Wh_2^{\mathbb{Z}}(\widehat{\mathbb{Z}}_p[G])$, the splitting $E_2 \circ \Gamma_2 = \text{id}$ is well-defined. This implies that $\Gamma_2$ is an injection, and combining this with the order argument, we conclude $\widetilde{K}_2^c(\widehat{\mathbb{Z}}_p[G]) \cong HC_1(\widehat{\mathbb{Z}}_p[G])$.
\end{proof}

\begin{remark} \label{K2L}
Let $G$ be a finite abelian $p$-group and $I$ be the augmentation ideal of $\widehat{\mathbb{Z}}_p[G]$. We first establish the linearization map $f: G \to \widehat{\mathbb{Z}}_p[G]$ by $f(\prod g_i^{\lambda_i}) = \sum \lambda_i g_i$ with $f(1)=0$. Based on this, we define the mapping $\varepsilon_2: \widetilde{H}_2(G) \to HC_1(\widehat{\mathbb{Z}}_p[G])$ by $\varepsilon_2(g \tilde{\wedge} h) = g \otimes g f(h)$. The linearity of $\varepsilon_2$ in the second variable is manifest, while for the first variable, it follows from the relation $g \otimes \lambda g = 0$ in $HC_1$, which implies the antisymmetric property $g_1 \otimes \lambda g_2 + g_2 \otimes \lambda g_1 = 0$. Specifically:
\begin{align*}
\varepsilon_2((g_1 g_2) \tilde{\wedge} h) &= (g_1 g_2) \otimes (g_1 g_2) f(h) \\
&= g_1 \otimes g_1 f(h) + g_2 \otimes g_2 f(h) + (g_1 \otimes g_2 f(h) + g_2 \otimes g_1 f(h)) \\
&= \varepsilon_2(g_1 \tilde{\wedge} h) + \varepsilon_2(g_2 \tilde{\wedge} h)
\end{align*}
where the cross-terms vanish identically by the "adding zero" technique. Since $\omega_2(\varepsilon_2(g \tilde{\wedge} h)) = g \tilde{\wedge} h$, the mapping $\varepsilon_2$ is a section of $\omega_2$, yielding an alternative proof that the exact sequence \eqref{Wh2} splits.

To extend the classical $p$-adic logarithm $\Gamma_2$ in \eqref{Wh2} to the full torsion subgroup, we recall the structural decomposition of the unit group. For a finite abelian $p$-group $G$, the group $SK_1(\widehat{\mathbb{Z}}_p[G])$ vanishes and the unit group $(\widehat{\mathbb{Z}}_p[G])^*$ admits a structural splitting based on the theory of integral logarithms (see \cite{oliver1988whitehead}):
\begin{equation} \label{unit-decomp}
(\widehat{\mathbb{Z}}_p[G])^* \cong \widehat{\mathbb{Z}}_p^* \times G \times Wh_1(\widehat{\mathbb{Z}}_p[G]),
\end{equation}
where $Wh_1(\widehat{\mathbb{Z}}_p[G]) \cong (1+I)/G$ is the torsion-free part.
Let $\mu_{m}(\widehat{\mathbb{Z}}_p)$ denotes the subgroup of $m$-th roots of unity in $\widehat{\mathbb{Z}}_p$.  The scalar factor is decomposed as
$$\widehat{\mathbb{Z}}_p^* \cong \begin{cases} \mu_{p-1}(\widehat{\mathbb{Z}}_p) \times (1+p\widehat{\mathbb{Z}}_p), & \text{if } p > 2; \\ \mu_{2}(\widehat{\mathbb{Z}}_2) \times (1+4\widehat{\mathbb{Z}}_2), & \text{if } p = 2. \end{cases}$$
Any unit $u$ is thus uniquely decomposed as $u = \zeta_u \cdot s_u \cdot h_u \cdot v_u$, with $\zeta_u$ in the torsion part of $\widehat{\mathbb{Z}}_p^*$, $s_u$ in the uniquely $p$-divisible part of $\widehat{\mathbb{Z}}_p^*$, $h_u \in G$, and $v_u \in Wh_1(\widehat{\mathbb{Z}}_p[G])$.

For an odd prime $p$, the symbols $\{g, \zeta_u\}$ and $\{g, s_u\}$ vanish because the order of $\zeta_u$ is coprime to $p$, and $1+p\widehat{\mathbb{Z}}_p$ is uniquely $p$-divisible. For $p=2$, the only non-trivial scalar contribution is $\{g, -1\} = \{g, g\}$. By setting $h'_u = h_u$ if $\zeta_u = 1$, and $h'_u = g h_u$ if $\zeta_u = -1$, the extended Oliver's logarithm $\widetilde{\Gamma}_2: \widetilde{K}_2^{c}(\widehat{\mathbb{Z}}_p[G]) \rightarrow HC_1(\widehat{\mathbb{Z}}_p[G])$ is defined as:
\begin{equation} \label{Gamma_2_ext}
\widetilde{\Gamma}_2(\{g, u\}) = g \otimes \Gamma_{G}(u) + \varepsilon_2(g \tilde{\wedge} h'_u) = g \otimes (\Gamma_{G}(u) + g f(h'_u)).
\end{equation}
Since $\widetilde{\Gamma}_2$ is the sum of the original $\Gamma_2$ in \eqref{Wh2} and the homomorphism $\varepsilon_2$, it is a well-defined isomorphism between finite groups of the same order. 

By virtue of the isomorphisms in \eqref{isos2}, the explicit map from $\widetilde{K}_2^c(\widehat{\mathbb{Z}}_p[G])$ to $HC_1(\widehat{\mathbb{Z}}_p[G])$ described above induces, as a natural consequence, an isomorphism to the linearized $K_2$-group:
$$\widetilde{L}_2: \widetilde{K}_2^{c}(\widehat{\mathbb{Z}}_p[G]) \longrightarrow \widetilde{K}_{2,L}(\widehat{\mathbb{Z}}_p[G]).$$
Specifically, this map is given on symbols by
$$\widetilde{L}_2(\{g, u\}) = [g, \Gamma_{G}(u) + gf(h'_u)],$$ 
where the symbol $[\cdot, \cdot]$ follows the convention in \cite{clauwens1987k}.
\end{remark}

The following examples are obtained through direct computation. The first case is due to Oliver (see \cite[Lemma 3.8]{oliver1987k}). For the counting of cyclic subgroups of a specific $p$-power order in $G$, one may refer to \cite{yeh1948on}. Here, $\varphi$ denotes the Euler's totient function.

\begin{eg}
Let $k \geq 1$ and $G = (C_{p^n})^k$. Then
$$\widetilde{K}_2^{c}(\widehat{\mathbb{Z}}_p[G]) \cong \left( \bigoplus_{i=1}^{n-1} (\mathbb{Z}/p^i\mathbb{Z})^{(p^k-1)p^{k(n-i-1)}} \right) \oplus (\mathbb{Z}/p^n\mathbb{Z})^{1+(k-1)p^{kn}}.$$
\end{eg}

\begin{eg}
Let $k, n \geq 1$, $G = (C_p)^k \times C_{p^n}$, and $r = p^{k+1} - p + 1$. Then
$$\widetilde{K}_2^{c}(\widehat{\mathbb{Z}}_p[G]) \cong (\mathbb{Z}/p\mathbb{Z})^{p^n(1+(k-1)p^k)} \oplus \left( \bigoplus_{i=1}^{n-1} (\mathbb{Z}/p^i\mathbb{Z})^{r\varphi(p^{n-i})} \right) \oplus (\mathbb{Z}/p^n\mathbb{Z})^{r}.$$
\end{eg}

\section*{Acknowledgement}

We are deeply indebted and sincerely grateful to the referees for their valuable time and insightful comments. 

\section*{Disclosure statement}
The author declares that there are no competing interests of a financial or personal nature.

\bibliography{mybibfile202408}

\end{document}